\documentclass[reqno]{amsart}
\usepackage{amsmath,graphicx,amsfonts,amssymb,amsthm,hyperref,amscd,extarrows,enumitem,tikz-cd}
\usepackage{color}
\usepackage[all,cmtip]{xy}
\providecommand{\U}[1]{\protect\rule{.1in}{.1in}}
\def\theenumi{\arabic{enumi}}

\def\theenumii{\alph{enumii}}
\def\p@enumii{\theenumi.}

\def\theenumiii{\arabic{enumiii}}
\def\p@enumiii{(\theenumi)(\theenumii)}

\def\p@enumiv{\p@enumiii.\theenumiii}
\DeclareMathOperator{\F}{Frob}

\DeclareMathOperator{\Aut}{Aut}

\DeclareMathOperator{\Sp}{Spec}

	\newtheorem{theorem}{Theorem} 
	\newtheorem*{theorem*}{Theorem} 
	\newtheorem{lema}[theorem]{Lemma}
	\newtheorem*{lema*}{Lemma}
	 
	\newtheorem*{corollary*}{Corollary} 
	\theoremstyle{definition}				
	\newtheorem*{conjecture*}{Conjecture} 
	
	\newtheorem*{proposition*}{Proposition}
	\newtheorem*{problem*}{Problem}
	
	\newtheorem*{example*}{Example}
	\newtheorem{definition}[theorem]{Definition} 
	\newtheorem*{definition*}{Definition} 
	\newtheorem*{notation*}{Notation} 
	\newtheorem*{remark*}{Remark} 
	\newtheorem*{exercise*}{Exercise}

\begin{document}
	
	\title{ Average size of the automorphism group of smooth projective hypersurfaces over finite fields}

	\date{}
\author{
Vlad Matei
}
\address[Vlad Matei]{
\begin{itemize}
\item[-]
Department of Mathematics University of California Irvine, 340 Rowland Hall, Irvine, CA, 92697
\end{itemize}
}
 \email[Vlad Matei]{vmatei@math.uci.edu}

	\maketitle
		
	\begin{abstract}
	
	In this paper we show that the average size of the automorphism group over $\mathbb{F}_q$ of a smooth degree $d$ hypersurface in $\mathbb{P}^{n}_{\mathbb{F}_q}$ is equal to $1$ as $d\rightarrow \infty$. We also discuss some consequence of this result for the moduli space of smooth degree $d$ hypersurfaces in $\mathbb{P}^n$.
	
	\end{abstract}
	
	\section{ Introduction }

\begin{definition*} Let $\mathcal{S}_{n,d}$ denote the set of smooth degree $d$ hypersurfaces in $\mathbb{P}^{n}_{\mathbb{F}_q}$.

\end{definition*}

The automorphism group of a projective smooth degree $d$ hypersurface $X$, $\Aut(X)$,  in $\mathbb{P}^{n}_{\mathbb{F}_q}$  over the algebraic closure $\overline{\mathbb{F}_q}$   has been an object of intense  study over the past decades. It is known that for most $(n,d)$, i.e $(n,d)\neq (2,3),(3,4)$ all of the automorphisms of the  hypersurface $X$ are induced by by an automorphism $\mathbb{P}^{n}_{\overline{\mathbb{F}_q}}$. This has been proven  by Matsumura and Monsky in \cite{mm} for $n\geq 2$ and Chang in \cite{ch} for $n=1$.

Another classical fact is that $\Aut(X)$ is finite for $n\geq 2, d\geq 3$. This was shown for $n\geq 3$ by Matsumura and Monsky in \cite{mm}. For $n=2$ and $d\geq 4$, the genus of a smooth degree $d$ curve is $g=(d-1)(d-2)/2$ is a least $2$ so we can use an old result of Schmid \cite{sc}. Finally, the argument for $d=3$ can be found in \cite{po3}.

We also know that for $(n,d)\neq (2,3)$ there is a open subset $U_{n,d}$ of  $\mathcal{S}_{n,d}$ such that $\Aut(X)=\{1\}$; here we include $(n,d)=(3,4)$ with the caveat that we are only considering linear automorphisms, i.e., induced by automorphisms of $\mathbb{P}^n$. This result can be found in work of Katz and Sarnak \cite[Lemma 11.8.5]{ka}. The fact that $U_{n,d}$ is nonempty follows from work of \cite{mm} for $n\geq 3,d\geq 3$, and this can be adapted to work for the cases $n\geq 2, d\geq 4$. An alternative proof of the case $n=2$ and $d\geq 4$ can be found in \cite[Lemma 10.6.18]{ka}. 

For our purposes we are only going restrict to the set of automorphisms defined over the base field of the hypersurface, $\mathbb{F}_q$, and prove a quantitative version of the above.

\begin{theorem}

We have  that 

$$\lim_{d\rightarrow \infty} \frac{\displaystyle\sum_{X\in \mathcal{S}_{n,d}} |\Aut_{\mathbb{F}_q}(X)|}{|\mathcal{S}_{n,d}|}=1.$$

.

\end{theorem} 

We also have an error term version of Theorem 1 that will be used in the last section on the moduli space of smooth degree $d$ hypersurfaces in $\mathbb{P}^n$.

\begin{theorem}

For  $(n,d)\neq (2,3), (3,4)$ the number of smooth  degree $d$ hypersurfaces $X$ in $\mathbb{P}^n$ with $\Aut_{\mathbb{F}_q}(X)\neq \{1\}$ is less than $q^{C\binom{n+d}{d}+(n+1)^2}$, where  $C=1-\cfrac{1}{2^n}$.

\end{theorem}

\begin{remark*} The above theorem can be used to obtain stronger version of Theorem 1. Namely one show that the average size of $|\Aut_{\mathbb{F}_{q^r}}(X)|$  is $1$ for any $r\leq C_1 \binom{n+d}{d}$ where $C_1$ is a fixed constant. 

\end{remark*}

\bigskip

\textsc{Acknowledgments} This research was started at the MRC ``Explicit methods in Arithmetic Geometry in Characteristic $p$". The author would like to thank the organizers: Renee H. Bell, Valentijn Karemaker, Padmavathi Srinivasan, and Isabel Vogt.  This material is based upon work supported by the National Science Foundation under Grant Number DMS 1641020. The author would like to thank Bjorn Poonen for suggesting this project. The author would like also to thank Nathan Kaplan and Jesse Wolfson for suggestions and comments.
	
	\section{Proof of the main theorems}

	We will not write anymore explicitly the restriction $(n,d)\neq (2,3), (3,4)$, but the reader should bear it in mind for all the proofs. We start with the following result which appears in  \cite{po1}:
	
	\begin{theorem}[Poonen]  We have that 
	
	$$\lim_{d\rightarrow \infty} \frac{|\mathcal{S}_{n,d}|}{q^{\binom{d+n}{n}}}=\zeta_{\mathbb{P}^n}(n+1) .$$

	\end{theorem}
	
	This implies that a positive proportion of hypersufaces  are smooth once $d$ is large enough compared to $q$. 

	We are left to to estimate $\displaystyle \sum_{X\in \mathcal{S}_{n,d}} |\Aut_{\mathbb{F}_q}(X)|$.  We know that any automorphism of the hypersurface $X$ is induced from $\operatorname{PGL}_{n+1}$. Thus we can write
	
	$$\displaystyle \sum_{X\in \mathcal{S}_{n,d}} |\Aut_{\mathbb{F}_q}(X)|=\sum_{X\in \mathcal{S}_{n,d}}\sum_{\substack{A\in \operatorname{PGL}_{n+1}(\mathbb{F}_q)\\ A\in \Aut(X)}} 1=\sum_{A\in \operatorname{PGL}_{n+1}(\mathbb{F}_q)}\sum_{\substack{X\in \mathcal{S}_{n,d}\\ A \in \Aut(X) }} 1$$ .
	
	Our estimates will not take into account the smoothness of the curves and rather prove a general bound. This leads to introducing the following definition and notations and note that we have be careful about scaling the matrix, since we will pick a representative for each element of $\operatorname{PGL}_{n+1}$ in $\operatorname{GL}_{n+1}$.

	\begin{definition}Let $\mathcal{P}_d$ the vector space of homogenous polynomials of degree $d$ in $\mathbb{F}_q[x_1,x_2,\ldots,x_{n+1}]$.  
	
	For an element $A=(a_{i,j})\in \operatorname{GL}_{n+1}(\mathbb{F}_q)$ and $\lambda\in\mathbb{F}_q$ consider
	
	$$\mathcal{P}^{A,\lambda}=\{f \in \mathcal{P}_d\mid \lambda f=f\circ A\}$$
	
	where $f\circ A$ denotes the polynomial where the variables $(x_1,x_2,\ldots,x_{n+1})$ are changed to $A(x_1,x_2,\ldots, x_{n+1})^t$. 
	
	\end{definition}
	
	It is easy to see that $\mathcal{P}^{A,\lambda}$ is a linear subspace. The key to the proof is the following estimate
	
	\begin{theorem}
	
	If $A\neq \lambda I_{n+1}$ then $\dim( \mathcal{P}^{A,\lambda})\leq \dbinom{d+n}{n}-\dbinom{d-\lfloor d/2 \rfloor +n}{n}$.
	
	\end{theorem} 
	
	\bigskip
	
	\begin{remark*} This  inequality is not sharp. We believe the bound in Lemma 7 to be closer to the truth. The proof of Lemma 7 suggests that diagonal automorphisms defined on the generators by $x_i\rightarrow -x_i, x_j\rightarrow x_j$ for $j\neq i$, and $1\leq i\leq n$  should give the right estimate. The above  bound is asymptotically equal to $\left(1-\cfrac{1}{2^n}\right)\dbinom{d+n}{n}$ and the bound in Lemma 7 is asymptotically $\left(1-\cfrac{d}{2(d+n)}\right)\dbinom{d+n}{n}$.
	
	\end{remark*}

	\bigskip
	\begin{lema} If $\dim( \mathcal{P}^{A,\lambda})\geq \dbinom{d+n}{n}-\dbinom{d-\lfloor d/2 \rfloor +n}{n}+1$ then $A$ must be a diagonal matrix.
	
	\end{lema}
	
	\begin{proof}
Let $\mathcal{P}^{r,x_i}$ be the vector space of homogenous polynomials of degree $d$ that can be written as $x_i^rh(x_1,x_2,\ldots,x_{n+1})$. Note that the dimension of this vector subspace is $\dbinom{d-r+n}{n}$ since obviously the degree of $h$ must be $d-r$.

Then if $\dim(\mathcal{P}^{r,x_i})+\dim( \mathcal{P}^{A,\lambda})\geq \dbinom{d+n}{n}+1$ then their intersection must be nonempty. Thus there is a nonzero $f_i=x_i^rh_i(x_1,x_2,\ldots,x_{n+1})$ such that we have 

$$x_i^rh_i(x_1,\ldots,x_{n+1})=(a_{i,1}x_1+a_{i,2}x_2+\ldots+a_{i,n+1}x_{n+1})^r\cdot  (h_i\circ A).$$

Thus note that the valuation of $x_i$ on the right hand side must be at least $r$. For $r>d/2$ since $h_i\neq 0$ we have that the valuation of $x_i$ in $h$ is at most $d-r<r$. Thus we must have $x_i|a_{i,1}x_1+a_{i,2}x_2+\ldots+a_{i,n+1}x_{n+1}$ so $a_{i,k}=0$ for $k\neq i$. We want to maximize the dimension of $\mathcal{P}^{r,x_i}$ to given an upper bound for $ \dim( \mathcal{P}^{A,\lambda})$ so we choose $r=\lfloor d/2\rfloor+1$.

This finishes the proof of the lemma.

\end{proof}

\begin{lema} 

If $A$ is a diagonal matrix  and $A\neq \lambda I_{n+1}$ then $\dim( \mathcal{P}^{A,\lambda})\leq \cfrac{1}{2}\dbinom{d+n-1}{n-1}+\cfrac{1}{2}\dbinom{d+n}{n}$.
\end{lema}

	\begin{proof} 
	

	
 Let $\lambda_1,\lambda_2,\ldots,\lambda_{n+1}$ be the  nonzero diagonal entries in the matrix $A$. Then by identifying coefficients on both sides of $f=f\circ A$  we have $f_{j_1,j_2,\ldots,j_{n+1}}=0$ or $\lambda_1^{j_1}\lambda_2^{j_2}\ldots\lambda_{n+1}^{j_{n+1}}=1$, where $j_1+\ldots+j_{n+1}=d$ . 

From this it follows that the dimension of $\mathcal{P}^{A,\lambda}$ is the number of tuples of indexes $(j_1,\ldots,j_{n+1})$ with the property that $\lambda_1^{j_1}\lambda_2^{j_2}\ldots\lambda_{n+1}^{j_{n+1}}=1$ and $j_1+\ldots+j_{n+1}=d$.

Note that  we can rewrite this multiplicative relation as $a_1^{j_1}a_2^{j_2}\ldots a_n^{j_n}=\lambda_{n+1}^{d}$ where $a_i=\cfrac{\lambda_i}{\lambda_{n+1}}$ for $1\leq i\leq n$. The condition in the statement can be translated to not allowing all of the $a_i$'s to be $1$.


We can assume wlog that $a_1\neq 1$. This means $a_1$ has order at least $2$, so for a fixed tuple $(j_2,\ldots, j_n)$ with $j_2+\ldots+j_n\leq d$ there are at most $\cfrac{d-(j_2+\ldots+j_n)}{2}+1$ values of $j_1$we can pick such that the stated equality holds. Thus the number of such indexes is at most

$$\sum_{ j_2+\ldots+j_n\leq d} \left(\cfrac{d-(j_2+\ldots+j_n)}{2}+1\right)=\sum_{k=0}^{d} \left(\frac{d-k}{2}+1\right)\dbinom{k+n-2}{n-2}$$

. Now all we have to observe is that the following two identities hold $\displaystyle \sum_{k=0}^{d} \dbinom{k+n-2}{n-2}=\dbinom{d+n-1}{n-1}$ and $\displaystyle \sum_{k=0}^{d}(k+n-1) \dbinom{k+n-2}{n-2}=(n-1)\sum_{k=0}^{d}\dbinom{k+n-1}{n-1}=(n-1)\dbinom{d+n}{n}.$

Thus the number of indexes $(j_1,\ldots, j_n)$ is at most 

$$\cfrac{d+n+1}{2}\dbinom{d+n-1}{n-1}-\cfrac{(n-1)}{2}\dbinom{d+n}{n}=\cfrac{1}{2}\dbinom{d+n-1}{n-1}+\cfrac{1}{2}\dbinom{d+n}{n}.$$

This finishes the proof of the lemma.

	\end{proof}

Note that lemma 6 and lemma 7 together immediately imply theorem 5. We are ready to prove now the main theorems.

\bigskip

\textbf{Proof  of Theorems 1 and 2}

\bigskip

It is easy to observe  that
	
	$$ |\mathcal{S}_{n,d}|\leq \sum_{X\in \mathcal{S}_{n,d}} |\Aut_{\mathbb{F}_q}(X)|\leq |\mathcal{S}_{n,d}|+\sum_{(A,\lambda), A\neq \lambda I_{n+1} } q^{\dim( \mathcal{P}^{A,\lambda})}.$$

	Now for the upper bound since the size of $\operatorname{GL}_{n+1}(\mathbb{F}_q)=(q^{n+1}-1)(q^{n+1}-q)\ldots(q^{n+1}-q^n)$ we have 
	
	$$\sum_{(A,\lambda), A\neq \lambda I_{n+1} } q^{\dim( \mathcal{P}^{A,\lambda})}\leq q^{\binom{d+n}{n}-\binom{d-\lfloor d/2 \rfloor+n-1}{n}}(q^{n+1}-1)\ldots (q^{n+1}-q^n)<q^{\binom{d+n}{n}-\binom{d-\lfloor d/2 +\rfloor +n-1}{n}+(n+1)^2}.$$
	
	Thus we get $\displaystyle |\mathcal{S}_{n,d}|\leq \sum_{X\in \mathcal{S}_{n,d}} |\Aut_{\mathbb{F}_q}(X)|< |\mathcal{S}_{n,d}|+q^{\binom{d+n}{n}-\binom{d-\lfloor d/2 \rfloor+n-1}{n}+(n+1)^2}$ and using Proposition 2 we are done.
	
	For theorem 2 just note that  
	
	$$ \sum_{X\in \mathcal{S}_d} |\Aut_{\mathbb{F}_q}(X)|=\sum_{X\in \mathcal{S}_{n,d},\Aut_{\mathbb{F}_q}(X)=\{1\} } 1+ \sum_{X\in \mathcal{S}_{n,d},\Aut_{\mathbb{F}_q}(X)\neq \{1\} } |\Aut_{\mathbb{F}_q}(X)|$$

	$$\geq \sum_{X\in \mathcal{S}_{n,d},\Aut_{\mathbb{F}_q}(X)=\{1\} } 1+ \sum_{X\in \mathcal{S}_{n,d},\Aut_{\mathbb{F}_q}(X)\neq \{1\} } 2= |\mathcal{S}_{n,d}|+ \#\{ X\in \mathcal{S}_{n,d}|\Aut_{\mathbb{F}_q}(X)\neq \{1\}\}.$$

\hfill $\square$

\section{Outlook on the moduli space of smooth hypersurfaces}

The goal is to discuss some of the implications of theorem $2$. First let us introduce some notation. Let $\mathcal{H}_{n,d}$ be the moduli space of degree $d$ hypersurfaces in $\mathbb{P}^n$; more precisely, the quotient $\mathcal{S}_{n,d}/\operatorname{PGL}_{n+1}$. It is well known that it  is a Deligne-Mumford stack over $\Sp(\mathbb{Z})$; see \cite[Theorem 1.6]{be}.
We also have the following fiber bundle 

\[
\begin{tikzcd}
  \operatorname{PGL}_{n+1} \ar[r] &\mathcal{S}_{n,d} \ar[d] \\
  & \mathcal{H}_{n,d}
\end{tikzcd}
\]

which translates into the equality $[\operatorname{PGL}_{n+1}]\cdot [\mathcal{H}_{n,d}]= [\mathcal{S}_{n,d}]$ in the Grothendieck ring of stacks, $K_0(\operatorname{St})$.

The set $|\mathcal{H}_{n,d}(\mathbb{F}_q)|$ encodes the number isomorphism classes of smooth projective hypersurfaces of degree $d$ in $\mathbb{P}^n$. We can pick one representative for each of the isomorphism classes and thus identify  this with $|\mathcal{H}_{n,d}(\mathbb{F}_q)|$. Note that we have a natural probability measure on this where for each hypersuface we weight it by $\cfrac{1}{|\Aut_{\mathbb{F}_q}(X)|}$. This is well understood to be most natural way to count objects which have automorphisms and we obviously get the equality $|\mathcal{H}_{n,d}(\mathbb{F}_q)|=\displaystyle \sum\cfrac{1}{|\Aut_{\mathbb{F}_q}(X)|} $, where the sum is taken over the representatives we've chosen.

Using theorem 2 we obtain the following: $|\mathcal{H}_{n,d}(\mathbb{F}_q)|=\cfrac{|S_{n,d}|}{|\operatorname{Gl}_{n+1}(\mathbb{F}_q)|}+O(q^{C\binom{n+d}{d}+(n+1)^2})$, where  $C=1-\cfrac{1}{2^n}$.
We can see this as an incarnation of the above equality in the Grothendieck ring of stacks, but this is neither implied, or implies the equality in  $K_0(\operatorname{St})$.

We will use a special case of  \cite[Theorem 1.2]{bu} which gives an error term for Bertini theorems over finite fields. 

\begin{theorem}[Bucur-Kedlaya] We have that 
$$\left|\frac{|\mathcal{S}_{n,d}|}{q^{\binom{d+n}{n}}}-\zeta_{\mathbb{P}^n}(n+1)\right|\leq C_1 q^{-\delta}+C_2 q^{-\frac{d}{\max\{n+1,p\}}}$$

where $C_1$ and $C_2$ are explicit constants and $1+\cfrac{\log_q(d)}{n}>\delta>\cfrac{\log_q(d)}{n}-2$.

\end{theorem}

This implies that for $d$ large enough we have the following theorem.

\begin{theorem} The number isomorphism classes of smooth projective hypersurfaces of degree $d$ in $\mathbb{P}^n$ is

$$|\mathcal{H}_{n,d}(\mathbb{F}_q)|=q^{D}-q^{D-1}+O(q^{D-\delta})$$

where $D=\dbinom{d+n}{n}-(n+1)^2+1=\dim(\mathcal{H}_{n,d})$.

\end{theorem}

Let us make a few remarks about the above result. We can think about the above as being a stabilization result for the point counts on the moduli space $\mathcal{H}_{n,d}$, and it offers a different perspective on the stabilization results proved in \cite{wo}. 
We can also try to count the number of points on $\mathcal{H}_{n,d}$  using the Lefschetz trace formula for Deligne Mumford Stacks proved by Behrend \cite{beh}. Tomassi \cite{to} has proven that the singular  cohomology  of $\mathcal{H}_{n,d}$ vanishes in degrees $2\leq k\leq \cfrac{d+1}{2}$, thus stabilizing.  We know that $\mathcal{H}_{n,d}$ has a compactification such that the boundary is a normal crossings divisor with respect to $\Sp(\mathbb{Z})$, and thus we can compare the etale cohomology and the singular cohomology and conclude that etale cohomology vanishes also in these degrees.
As far as we know there are no results proven about the unstable cohomology $\mathcal{H}_{n,d}$, so the above theorem cannot be simply deduced from this cohomological side. 

If we use the same heuristic as in \cite{ac} about the eigenvalues of $\F$ acting on $H^{*}_{et}$, we can conjecture that there should not be too many unstable classes coming from algebraic cycles in each degree close to the dimension $2D$. It is well known that the eigenvalues of Frobenius attached to algebraic cycles are integral powers of $q$.  The classes that are not coming from algebraic cycles of weight close to the dimension $2D$ can either be: few in each weight so their contribution is small, or they  can be modeled using a random unitary matrix, and using a result of Diaconis and Shahshahani \cite{di} this matrix has bounded trace with high probability.  Thus we can heuristically think about the contribution of non-algebraic classes in the Grothendieck-Lefschetz trace formula  as being negligible as $d\rightarrow \infty$. 

{\small{

}}
	
	\end{document}